\documentclass[a4paper,12pt,reqno]{amsart}

\usepackage{hyperref,color}
\usepackage{amsmath,amsfonts}
\usepackage{amsthm, bbm}

\newtheorem{lemma}{Lemma}[section]
\newtheorem{theorem}[lemma]{Theorem}
\newtheorem{corollary}[lemma]{Corollary}

\newtheorem{definition}[lemma]{Definition}
\newtheorem{example}[lemma]{Example}

\date{}

\begin{document}

\title[Conditionally  centred moments and compact operators on $L^p$]{Sharp estimates for conditionally  centred moments 
and for compact operators on  $L^p$ spaces}

\author[E. Shargorodsky]{Eugene Shargorodsky}
\address{
Department of Mathematics\\
King's College London\\
Strand, London WC2R 2LS\\
United Kingdom\quad
and\quad
Technische Universit\"at Dresden\\
Fakult\"at Mathematik\\
01062 Dresden\\
Germany}
\email{eugene.shargorodsky@kcl.ac.uk}

\author[T. Sharia]{Teo Sharia}
\address{
Department of Mathematics\\
Royal Holloway\\
University of London\\
Egham, Surrey TW20 0EX\\
United Kingdom}
\email{t.sharia@rhul.ac.uk}

\begin{abstract}
Let $(\Omega, \mathcal{F}, \mathbf{P})$ be a probability space, $\xi$ be a random variable on $(\Omega, \mathcal{F}, \mathbf{P})$, 
$\mathcal{G}$ be a sub-$\sigma$-algebra of $\mathcal{F}$,
and let $\mathbf{E}^\mathcal{G} = \mathbf{ E}(\cdot | \mathcal{G})$ be the corresponding conditional expectation operator.
We obtain sharp estimates for the moments of $\xi - \mathbf{E}^\mathcal{G}\xi$ in terms of the moments of $\xi$. This allows us
to find the optimal constant in the bounded compact approximation property of $L^p([0, 1])$, $1 < p < \infty$.
\end{abstract}

\subjclass[2000]{Primary 60E15, 47A30; Secondary 46B20, 46B28, 46E30, 47B07.}

\maketitle

\section{Introduction: estimates for centred moments}\label{intro}

Let $(\Omega, \mathcal{F}, \mathbf{P})$ be a probability space, $\xi$ be a real valued random variable (r.v.), and $\mathbf{E}\xi$ be the expectation of
$\xi$, i.e. $\mathbf{E}\xi := \int_\Omega \xi(\omega)\, d\mathbf{P}(\omega)$. We will always assume that $\Omega$ consists of more
than one element and that $\mathcal{F}$ is nontrivial, i.e. $\mathcal{F} \not= \{\emptyset, \Omega\}$.
It seems natural to ask what the optimal estimate of the centred $p$-th 
moment of $\xi$ by its $p$-th absolute (uncentred) moment is. In other words, one is looking for the optimal constant 
$c_p = c_p(\Omega, \mathcal{F}, \mathbf{P})$ in the estimate
\begin{equation}\label{1}
\|\xi - \mathbf{E}\xi\|_p = \left(\mathbf{E}|\xi - \mathbf{E}\xi|^p\right)^{1/p} \le c_p \left(\mathbf{E}|\xi|^p\right)^{1/p} = c_p \|\xi\|_p ,
\end{equation}
where $1 \le p < \infty$. For $p = \infty$, the above inequality takes the form
\begin{equation}\label{2}
\|\xi - \mathbf{E}\xi\|_\infty = \mathrm{ess}\sup_{\omega \in \Omega}|\xi(\omega) - \mathbf{E}\xi| \le 
c_\infty\, \mathrm{ess}\sup_{\omega \in \Omega}|\xi(\omega)| = c_\infty \|\xi\|_\infty .
\end{equation}
It follows from H\"older's inequality that $|\mathbf{E}\xi| \le \mathbf{E}|\xi| = \|\xi\|_1 \le \|\xi\|_p$. Hence
$$
\|\xi - \mathbf{E}\xi\|_p \le \|\xi\|_p +\|\mathbf{E}\xi\|_p = \|\xi\|_p +|\mathbf{E}\xi| \le 2\|\xi\|_p.
$$
So, $c_p \le 2$ for all $p \in [1, \infty]$. On the other hand, if $\xi$ is not a constant r.v., then $\eta := \xi - \mathbf{E}\xi \not= 0$, but
$\mathbf{E}\eta = 0$.  Hence $\|\eta- \mathbf{E}\eta\|_p = \|\eta\|_p$, and $c_p \ge 1$ for all $p \in [1, \infty]$.

It is well known that $c_2 = 1$. Indeed,
$$
\mathbf{E}|\xi - \mathbf{E}\xi|^2 = \mathbf{E}(\xi - \mathbf{E}\xi)^2 = \mathbf{E}\xi^2 - (\mathbf{E}\xi)^2 \le \mathbf{E}\xi^2 .
$$
Suppose that $(\Omega, \mathcal{F}, \mathbf{P})$ is nonatomic or, more generally, that for every $\alpha \in (0, 1)$, 
there exists $\xi$ such that $\mathbf{P}(\xi = 1) = \alpha$
and $\mathbf{P}(\xi = 0) = 1 - \alpha$. It is clear that $\mathbf{E}|\xi| = \mathbf{E}\xi = \alpha$, 
$\mathbf{E}|\xi - \mathbf{E}\xi| = 2\alpha(1 - \alpha)$, and
$$
\frac{\|\xi - \mathbf{E}\xi\|_1}{\|\xi\|_1} = 2(1 - \alpha) .
$$
Sending $\alpha$ to $0$, one concludes that $c_1 = 2$. Similarly, if $\mathbf{P}(\xi = 1) = \alpha$
and $\mathbf{P}(\xi = -1) = 1 - \alpha$, then $\mathbf{E}\xi = 2\alpha - 1$, $\|\xi\|_\infty = 1$, and
$\|\xi - \mathbf{E}\xi\|_\infty = \max\{2\alpha, 2(1 - \alpha)\}$. Sending $\alpha$ to $1$ or to $0$, one concludes that $c_\infty = 2$.
Putting the above information together, one gets
\begin{equation}\label{allelem}
c_1 = 2 = c_\infty , \quad\quad c_2 = 1 , \quad\quad  1 \le c_p \le 2 \quad\mbox{for all}\quad p \in (1, \infty) ,
\end{equation}
where the first equality holds if there are $A \in \mathcal{F}$ with arbitrarily small positive $\mathbf{P}(A)$, e.g., if $\mathbf{P}$ is nonatomic.

The constant $c_p$ is obviously the norm of the operator $\mathbf{ C} : L^p(\Omega, \mathbf{P}) \to L^p(\Omega, \mathbf{P})$,
\begin{equation}\label{OpC}
\xi \mapsto \mathbf{C}\xi :=\xi - \mathbf{E}\xi .
\end{equation}
Applying the Riesz-Thorin interpolation theorem (see, e.g., \cite[Theorem 1.1.1]{BL76}) to this operator,
one deduces from the equalities in \eqref{allelem} that 
\begin{equation}\label{Rol}
c_p \le 2^{\left|1 - \frac2p\right|} , \quad 1 < p < \infty
\end{equation}
(see \cite{R90}). 

Let $(\Omega, \mathcal{F}, \mathbf{P}) = ([0, 1], \mathcal{L}, \lambda)$, where $\lambda$ is the standard Lebesgue measure on $[0, 1]$
and $\mathcal{L}$ is the $\sigma$-algebra of Lebesgue measurable subsets of $[0, 1]$. (Equivalently, one can assume that
$\Omega$ is a complete separable metric space, $\mathcal{F} = \mathcal{B}$ is the Borel $\sigma$-algebra of $\Omega$, and $\mathbf{P}$
is nonatomic; see, e.g., \cite[Theorem 9.2.2]{B07}.) C. Franchetti proved (see \cite{F90}) that in this case,
\begin{equation}\label{cp}
c_p = c_p([0, 1], \mathcal{L}, \lambda) = 
\max_{0 < \alpha < 1} C_p(\alpha) =: C_p \quad\mbox{for all}\quad  1 < p < \infty ,
\end{equation}
where
\begin{equation}\label{Cp}
C_p(\alpha) := 
 \left(\alpha^{p - 1} + (1 - \alpha)^{p - 1}\right)^{\frac1p} \left(\alpha^{\frac1{p - 1}} + (1 - \alpha)^{\frac1{p - 1}}\right)^{1 - \frac1p} .
\end{equation}
A simple analysis  shows that
\begin{align}
& C_1 := \lim_{p \to 1 + 0} C_p = 2 = c_1 , \quad\quad  C_\infty :=\lim_{p \to \infty} C_p = 2 = c_\infty ,  \label{C1}\\
& C_2 = 1 = c_2 ,   \quad\quad  C_{p'} = C_p \quad\mbox{for}\quad p' = \frac{p}{p - 1}\, , \quad\quad
 C_p \le 2^{\left|1 - \frac2p\right|} \nonumber
\end{align}
(see  \cite{F90} and \eqref{allelem}). Given \eqref{cp}, the last inequality is the same as \eqref{Rol}, 
while the equality $C_{p'} = C_p$ follows also from the fact that the adjoint
of the operator $\mathbf{ C} : L^p(\Omega, \mathbf{P}) \to L^p(\Omega, \mathbf{P})$ (see \eqref{OpC}) is
$\mathbf{C} : L^{p'}(\Omega, \mathbf{P}) \to L^{p'}(\Omega, \mathbf{P})$, and hence
$$
c_{p'} = \|\mathbf{C}\|_{L^{p'}(\Omega, \mathbf{P}) \to L^{p'}(\Omega, \mathbf{P})} 
= \|\mathbf{C}^*\|_{L^{p'}(\Omega, \mathbf{P}) \to L^{p'}(\Omega, \mathbf{P})} 
= \|\mathbf{C}\|_{L^{p}(\Omega, \mathbf{P}) \to L^{p}(\Omega, \mathbf{P})} = c_p .
$$ 
One can also find the following explicit values in \cite{F90}:
$$
C_3 = \frac13 \left(17 + 7\sqrt{7}\right)^{1/3} = 1.0957\dots , \quad C_4 = \left(1 + \frac23\sqrt{3}\right)^{1/4} = 1.21156\dots
$$
The proof of \eqref{cp} in \cite{F90} was quite complicated. A much simpler alternative proof was produced by T.F. M\'ori (see \cite{M09}), who
was apparently unaware of C. Franchetti's work. The latter proof goes as follows. First, consider a r.v. $\xi_\alpha$ such that 
$\mathbf{P}(\xi_\alpha = -b) = 1 - \alpha$ and $\mathbf{P}(\xi_\alpha = 1 - b) = \alpha$, where
\begin{equation}\label{b}
b :=  \frac{\alpha^{\frac{1}{p - 1}}}{\alpha^{\frac{1}{p - 1}} + (1 - \alpha)^{\frac{1}{p - 1}}}\, .
\end{equation}
Then  
\begin{align*}
& \mathbf{ E}\xi_\alpha = 
\frac{\alpha (1 - \alpha)^{\frac{1}{p - 1}} - (1 - \alpha)\alpha^{\frac{1}{p - 1}}}{\alpha^{\frac{1}{p - 1}} + (1 - \alpha)^{\frac{1}{p - 1}}}\, , \quad
\mathbf{ E}|\xi_\alpha|^p = \frac{\alpha (1 - \alpha)}{\left(\alpha^{\frac{1}{p - 1}} + (1 - \alpha)^{\frac{1}{p - 1}}\right)^{p - 1}}\, ,  \\
& \mathbf{E}|\xi_\alpha - \mathbf{E}\xi_\alpha|^p = \alpha (1 - \alpha) \left(\alpha^{p - 1} + (1 - \alpha)^{p - 1}\right) ,  \\
& \frac{\left(\mathbf{E}|\xi_\alpha - \mathbf{E}\xi_\alpha|^p\right)^{1/p}}{\left(\mathbf{E}|\xi_\alpha|^p\right)^{1/p}} = 
 \left(\alpha^{p - 1} + (1 - \alpha)^{p - 1}\right)^{\frac1p} \left(\alpha^{\frac1{p - 1}} + (1 - \alpha)^{\frac1{p - 1}}\right)^{1 - \frac1p} ,
\end{align*}
which shows that $c_p \ge C_p$. It will be important to us in the next section that the opposite inequality
\begin{equation}\label{cCp}
c_p(\Omega, \mathcal{F}, \mathbf{P}) \le C_p
\end{equation}
holds for all probability spaces $(\Omega, \mathcal{F}, \mathbf{P})$. The main ingredient that makes T.F. M\'ori's proof easier than
that of C. Franchetti's is the observation that
every zero mean probability distribution on $\mathbb{R}$ is a mixture of distributions concentrated on two points and 
having zero mean (see \cite[Lemma 14.4]{K02}
for a beautiful elementary proof, which is attributed there to K.L. Chung). This allows one to reduce the proof of \eqref{cCp} to showing that
$$
\frac{\left(\mathbf{E}|\xi - \mathbf{E}\xi|^p\right)^{1/p}}{\left(\mathbf{E}|\xi|^p\right)^{1/p}} \le C_p
$$
holds for every r.v. $\xi$ that takes only two values. The latter is an elementary although not an entirely trivial calculation.

Yet another proof of \eqref{cp} was produced by G. Lewicki and L. Skrzypek (see \cite{LS16}), who were apparently unaware of 
T.F. M\'ori's work. They considered the case where $\Omega = \{1, \dots, n\}$ and $\mathbf{ P}$ is the uniform distribution:
$\mathbf{ P}(k) = \frac1n$, $k = 1, \dots, n$, and proved that for $n = 3, 4$ and for all sufficiently large $n$, one has
\begin{equation}\label{cpn}
c_p = \max\left\{C_p\left(\frac{k_1}{n}\right), \ C_p\left(\frac{k_2}{n}\right)\right\} ,
\end{equation}(see \eqref{Cp}), where 
$$
 k_1 := \max\left\{k \in \mathbb{N} : \ \frac{k}n  \le \alpha_p \right\} , \quad
 k_2 := \min\left\{k \in \mathbb{N} : \ \alpha_p  \le \frac{k}n < \frac{1}2\right\} ,
$$
and $\alpha_p \in (0, 1/6)$ is the unique point at which $C_p(\alpha)$ attains its global maximum in $[0, 1/2]$.
They also conjectured that the above result holds for all $n \in \mathbb{N}$. (If $n = 3, 4$, then the inequality 
$\frac{k}n < \frac{1}2$ means that $k = 1$, and \eqref{cpn} takes the form
$$
c_p = \frac{ \left((n - 1)^{p - 1} + 1\right)^{\frac1p} \left((n - 1)^{\frac1{p - 1}} + 1\right)^{1 - \frac1p}}{n}\, .
$$
This result was obtained originally in \cite{S09} for $n = 2, 3, 4$.) G. Lewicki and L. Skrzypek showed that $c_p$ in \eqref{cpn} tends to
$c_p$ in \eqref{cp} as $n \to \infty$ and recovered C. Franchetti's result.

As mentioned above,
$$
1 \le c_p(\Omega, \mathcal{F}, \mathbf{P}) \le C_p
$$
holds for all probability spaces $(\Omega, \mathcal{F}, \mathbf{P})$. It turns out that for every $c \in [1, C_p]$
there exist a probability space $(\Omega, \mathcal{F}, \mathbf{P})$ such that $c_p(\Omega, \mathcal{F}, \mathbf{P}) = c$.
Indeed, let $\Omega = \{-1, 1\}$,  $\mathbf{P}(-1) = 1 - \alpha$, $\mathbf{P}(1) = \alpha$. Then $c_p = C_p(\alpha)$
(see \cite{M09}), and one can choose $\alpha$ in such a way that $c_p = c$, since
$$
C_p(1/2) = 1 = \lim_{\alpha \to 0} C_p(\alpha) = \lim_{\alpha \to 1} C_p(\alpha) , \quad\quad
C_p = \max_{0 < \alpha < 1} C_p(\alpha) .
$$ 

All the above results remain true for complex valued random variables 
(\cite[Ch. IV, Miscellaneous theorems and examples, 13]{Zyg}, see also \cite{HK}, \cite{MS}).
In the next section, we extend them to conditional expectations and then use the obtained results in Section \ref{3approx} to
find the optimal constant in the bounded compact approximation property of $L^p([0, 1])$, $1 < p < \infty$.

\section{Estimates for conditionally  centred moments}\label{main}

Let $(\Omega, \mathcal{F}, \mathbf{P})$ be a probability space, $\mathcal{G}$ be a sub-$\sigma$-algebra of $\mathcal{F}$,
and let $\mathbf{E}^\mathcal{G} = \mathbf{ E}(\cdot | \mathcal{G})$ be the corresponding conditional expectation operator.
Then $\mathbf{E}^\mathcal{G} : L^p(\Omega, \mathcal{F}, \mathbf{P}) \to L^p(\Omega, \mathcal{G}, \mathbf{P})$, $1 \le p \le \infty$
is a contractive projection that preserves constants, i.e.
\begin{align*}
& \left\|\mathbf{E}^\mathcal{G}\xi\right\|_p \le \|\xi\|_p , \quad \mathbf{E}^\mathcal{G}\left(\mathbf{E}^\mathcal{G} \xi\right) = \mathbf{E}^\mathcal{G} \xi
\quad\mbox{for all}\quad \xi \in L^p(\Omega, \mathcal{F}, \mathbf{P}) , \\
& \mathbf{E}^\mathcal{G} \mathbbm{1} = \mathbbm{1} ,
\end{align*}
where $\mathbbm{1}(\omega) = 1$ a.s. (see, e.g., \cite[Section 2.1, Theorem 9 and Section 2.2, Theorem 1]{R95} or \cite[Lemma 6.1.1]{AK06}).
In fact, every contractive projection on $L^p(\Omega, \mathcal{F}, \mathbf{P})$, $p \in [1, \infty)\setminus\{2\}$ that preserves constants
is the conditional expectation operator $\mathbf{E}^\mathcal{G}$ for a certain sub-$\sigma$-algebra $\mathcal{G} \subseteq \mathcal{F}$
(see \cite{A66}, \cite{D65}, \cite{S66}, and \cite[Section 2.2, Theorem 6]{R95}). 

{\bf Remark.}\ The conditional expectation operator $\mathbf{E}^\mathcal{G}$ is a contractive projection on a wide class of Banach 
function spaces that includes all rearrangement invariant spaces, e.g. Orlicz and Lorentz spaces (see \cite{C83}, \cite[Theorem 2.a.4]{LT2},
and \cite[Ch. 2, Theorem 4.8]{BS88}). The survey paper \cite{R01} contains wealth of information on contractive projections in Banach 
function spaces and on their representability as conditional expectation operators.
An example of a Banach function space over $([0, 1], \mathcal{L}, \lambda)$, on which 
$\mathbf{E}^\mathcal{G}$ is unbounded for a certain sub-$\sigma$-algebra $\mathcal{G}$ can be found in \cite[Example 4.8]{PR13}.

We are interested in the best constant $c_p = c_p(\Omega, \mathcal{F}, \mathcal{G}, \mathbf{P})$ in the estimate
\begin{equation}\label{1cond}
\left\|\xi - \mathbf{E}^\mathcal{G}\xi\right\|_p  \le c_p \|\xi\|_p , \quad 1 \le p \le \infty ,
\end{equation}
i.e. in the norm of the operator $ I - \mathbf{E}^\mathcal{G} : L^p(\Omega, \mathcal{F}, \mathbf{P}) \to L^p(\Omega, \mathcal{F}, \mathbf{P})$,
\begin{equation}\label{1norm}
c_p(\Omega, \mathcal{F}, \mathcal{G}, \mathbf{P}) = 
\left\|I - \mathbf{E}^\mathcal{G}\right\|_{L^p(\Omega, \mathcal{F}, \mathbf{P}) \to L^p(\Omega, \mathcal{F}, \mathbf{P})} ,
\end{equation}
where $I$ is the identity operator. 

Similarly to Section \ref{intro}, one has the following:
$$
\left\|\mathbf{E}^\mathcal{G}\right\| = 1 \ \Longrightarrow \ \left\|I - \mathbf{E}^\mathcal{G}\right\| \le 2 .
$$
If $\mathcal{G} \not= \mathcal{F}$, then there exists a r.v.  $\xi$ such that $\eta := \xi - \mathbf{E}^\mathcal{G}\xi \not= 0$. Since
$\mathbf{E}^\mathcal{G}\eta = 0$, one has  $\|\eta- \mathbf{E}^\mathcal{G}\eta\|_p = \|\eta\|_p$, and 
$\left\|I - \mathbf{E}^\mathcal{G}\right\| \ge 1$ for all $p \in [1, \infty]$. For $p =2$, it follows from
$$
\mathbf{E}|\xi - \mathbf{E}^\mathcal{G}\xi|^2 
= \mathbf{E}|\xi|^2 -  \mathbf{E}|\mathbf{E}^\mathcal{G}\xi|^2 \le \mathbf{E}|\xi|^2 
$$
(see \cite[Theorem 6.1(vi)]{K02}) that $c_2(\Omega, \mathcal{F}, \mathcal{G}, \mathbf{P}) = 1$. The following result
is an analogue of \eqref{cCp}. 
\begin{theorem}\label{upper}
\begin{equation}\label{upperest}
c_p(\Omega, \mathcal{F}, \mathcal{G}, \mathbf{P}) \le C_p , \quad 1 \le p \le \infty
\end{equation} 
(see \eqref{cp}--\eqref{C1}, \eqref{1norm}).
\end{theorem}
\begin{proof}
Since $C_1 = 2 = C_\infty$, it follows from the above that one only needs to consider the case $1 < p < \infty$. 
Take any $\varepsilon > 0$ and any $\xi \in L^p(\Omega, \mathcal{F}, \mathbf{P})$. Since 
$\mathbf{E}^\mathcal{G}\xi \in L^p(\Omega, \mathcal{G}, \mathbf{P})$, it can be approximated by a countably valued 
$\mathcal{G}$-measurable function, i.e. 
there exist pairwise disjoint sets $A_n \in \mathcal{G}$
and numbers $a_n \in \mathbb{C}$, $n = 1, \dots, N$, $N \in \mathbb{N}\cup\{\infty\}$ such that $\mathbf{P}(A_n) > 0$,
$\cup_{n = 1}^N A_n = \Omega$, and
$$
\left|\mathbf{E}^\mathcal{G}\xi - \sum_{n = 1}^N a_n \mathbbm{1}_{A_n}\right| < \varepsilon \quad\mbox{a.s.} 
$$
Then 
\begin{align*}
& \left|\frac{1}{\mathbf{P}(A_n)} \int_{A_n} \xi\, d\mathbf{P} - a_n\right| 
= \left|\frac{1}{\mathbf{P}(A_n)} \int_{A_n} \mathbf{E}^\mathcal{G}\xi\, d\mathbf{P} - a_n\right| \\
& = \left|\frac{1}{\mathbf{P}(A_n)} \int_{A_n} \left(\mathbf{E}^\mathcal{G}\xi - a_n\right)\, d\mathbf{P}\right| 
 =  \left|\frac{1}{\mathbf{P}(A_n)} 
\int_{A_n} \left(\mathbf{E}^\mathcal{G}\xi - \sum_{k = 1}^N a_k \mathbbm{1}_{A_k}\right)\, d\mathbf{P}\right| \\
& \le \frac{1}{\mathbf{P}(A_n)} 
\int_{A_n} \left|\mathbf{E}^\mathcal{G}\xi - \sum_{k = 1}^N a_k \mathbbm{1}_{A_k}\right|\, d\mathbf{P} < \varepsilon .
\end{align*}
Hence
\begin{align*}
& \left\|\xi - \mathbf{E}^\mathcal{G}\xi\right\|_p  \le 
\left\|\xi - \sum_{n = 1}^N \left(\frac{1}{\mathbf{P}(A_n)} \int_{A_n} \xi\, d\mathbf{P}\right) \mathbbm{1}_{A_n}\right\|_p \\
&+ \left\|\sum_{n = 1}^N \left(\frac{1}{\mathbf{P}(A_n)} \int_{A_n} \xi\, d\mathbf{P}\right) \mathbbm{1}_{A_n}
- \sum_{n = 1}^N a_n \mathbbm{1}_{A_n}\right\|_p 
+ \left\|\sum_{n = 1}^N a_n \mathbbm{1}_{A_n} - \mathbf{E}^\mathcal{G}\xi\right\|_p \\
&< \left\|\xi - \sum_{n = 1}^N \left(\frac{1}{\mathbf{P}(A_n)} \int_{A_n} \xi\, d\mathbf{P}\right) \mathbbm{1}_{A_n}\right\|_p 
+ 2\varepsilon .
\end{align*}
Applying \eqref{cCp} to the probability spaces $\left(A_n, \mathcal{F}_n, \frac{1}{\mathbf{P}(A_n)} \mathbf{P}\right)$, where
$$
\mathcal{F}_n := \{A\cap A_n | \ A \in \mathcal{F}\},
$$
one gets
\begin{align*}
&\left\|\xi - \sum_{n = 1}^N \left(\frac{1}{\mathbf{P}(A_n)} \int_{A_n} \xi\, d\mathbf{P}\right) \mathbbm{1}_{A_n}\right\|_p^p \\
&= \sum_{n = 1}^N \int_{A_n} \left|\xi(\omega) - \frac{1}{\mathbf{P}(A_n)} \int_{A_n} \xi\, d\mathbf{P}\right|^p d\mathbf{P}(\omega) \\
& \le \sum_{n = 1}^N C_p^p \int_{A_n} \left|\xi(\omega)\right|^p d\mathbf{P}(\omega) = C_p^p \|\xi\|_p^p .
\end{align*}
So,
$$
\left\|\xi - \mathbf{E}^\mathcal{G}\xi\right\|_p \le C_p \|\xi\|_p + 2\varepsilon \quad\mbox{for all}\quad \varepsilon > 0 ,
$$
i.e.
$$
\left\|\xi - \mathbf{E}^\mathcal{G}\xi\right\|_p \le C_p \|\xi\|_p \quad\mbox{for all}\quad \xi \in L^p(\Omega, \mathcal{F}, \mathbf{P}) .
$$
\end{proof}

\begin{example}
{\rm One might ask whether $\mathbf{E}^\mathcal{G}\xi$ is a ``better approximation" of $\xi$ in the $L^p$ norm than 
$\mathbf{E}\xi$, i.e. whether the inequality $c_p(\Omega, \mathcal{F}, \mathcal{G}, \mathbf{P}) \le c_p(\Omega, \mathcal{F}, \mathbf{P})$
holds. The following example shows that, in general, this is not the case. Let $1 < p < \infty$,
$\alpha_p \in (0, 1)$ be a point at which 
$C_p(\alpha)$ attains its maximum (see \eqref{Cp}), $\Omega = \{-1, 0, 1\}$, $\mathbf{P}(-1) = \tau(1 - \alpha_p)$, 
$\mathbf{P}(1) = \tau\alpha_p$, $\mathbf{P}(0) = 1 - \tau$, $0 < \tau < 1$, and
$$
\mathcal{G} = \Big\{\emptyset, \{0\}, \{-1, 1\}, \Omega\Big\} .
$$
If $\tau$ is close to 0, then it is natural to expect $c_p(\Omega, \mathcal{F}, \mathbf{P})$ to be close to the constant $c_p$ for the
probability space consisting of two points with probabilities $1 - \tau$ and $\tau$, i.e. to $C_p(1 - \tau)$ and hence to 1 (see the end of 
Section \ref{intro}). On the other hand, if $\xi$ is a random variable supported by $\{-1, 1\}$, then $\mathbf{ E}^{\mathcal{G}}\xi$ is also
supported by the set $\{-1, 1\}$, where it is constant. Hence $c_p(\Omega, \mathcal{F}, \mathcal{G}, \mathbf{P})$ is greater than or
equal to the constant $c_p$ for the probability space consisting of two points $1$ and $-1$ with probabilities $\alpha_p$ and $1 - \alpha_p$, i.e. to $C_p(\alpha_p) = C_p$.
Then one has $c_p(\Omega, \mathcal{F}, \mathcal{G}, \mathbf{P}) = C_p$ due to Theorem \ref{upper}. Here is a more detailed argument.

For any r.v. $\xi$, one has
$$
\mathbf{E}\xi = \tau(1 - \alpha_p)\xi(-1) +  \tau\alpha_p \xi(1) + (1 - \tau)\xi(0) . 
$$
Consider the random variables $\eta$, $\zeta$, and $\xi_0$ defined as follows
\begin{align*}
& \eta(\pm 1) = \mathbf{E}\xi , \quad \eta(0) = 0 , \\
& \zeta(\pm 1) = 0 ,  \quad  \zeta(0) = \tau(1 - \alpha_p)\xi(-1) +  \tau\alpha_p \xi(1) , \\
& \xi_0(\pm 1) = 0 ,  \quad  \xi_0(0) = (1 - \tau)\xi(0) .
\end{align*}
It is easy to see that
\begin{align*}
& \|\eta\|_p = \tau^{1/p} |\mathbf{E}\xi| \le \tau^{1/p} \|\xi\|_p \\
&  \|\zeta\|_p = (1 - \tau)^{1/p} \tau |(1 - \alpha_p)\xi(-1) +  \alpha_p \xi(1)| \\
& \le (1 - \tau)^{1/p} \tau \left((1 - \alpha_p)|\xi(-1)|^p +  \alpha_p |\xi(1)|^p\right)^{1/p} \le (1 - \tau)^{1/p} \tau^{1 - 1/p} \|\xi\|_p , \\
& \eta(\omega) + \zeta(\omega) + \xi_0(\omega) = \mathbf{E}\xi \quad\mbox{for all}\quad \omega \in \Omega .
\end{align*}
Hence 
\begin{align*}
&\|\xi - \mathbf{E}\xi\|_p = \|\xi - (\eta + \zeta +\xi_0)\|_p \le \|\xi - \xi_0\|_p + \|\eta\|_p + \|\zeta\|_p \\
& \le \|\xi\|_p + \tau^{1/p} \|\xi\|_p + (1 - \tau)^{1/p} \tau^{1 - 1/p} \|\xi\|_p \le \left(1 + \tau^{1/p} + \tau^{1 - 1/p}\right) \|\xi\|_p .
\end{align*}
So, 
$$
c_p(\Omega, \mathcal{F}, \mathbf{P}) \le 1 + \tau^{1/p} + \tau^{1 - 1/p} ,
$$
and choosing a sufficiently small $\tau$, one can make $c_p(\Omega, \mathcal{F}, \mathbf{P})$ arbitrarily close to $1$. On the other
hand, let $\xi(-1) = -b$, $\xi(1) = 1 - b$, $\xi(0) = 0$, where $b$ is defined by \eqref{b} with $\alpha = \alpha_p$. Then the same calculations
as in Section \ref{intro} show that 
\begin{align*}
& \mathbf{ E}^{\mathcal{G}}\xi(\pm1) = 
\frac{\alpha_p (1 - \alpha_p)^{\frac{1}{p - 1}} - (1 - \alpha_p)\alpha_p^{\frac{1}{p - 1}}}{\alpha_p^{\frac{1}{p - 1}} + (1 - \alpha_p)^{\frac{1}{p - 1}}}\, , 
\quad \mathbf{ E}^{\mathcal{G}}\xi(0) = 0 ,  \\
& \frac{\left\|\xi - \mathbf{E}^{\mathcal{G}}\xi\right\|_p}{\|\xi\|_p} = 
 \left(\alpha_p^{p - 1} + (1 - \alpha_p)^{p - 1}\right)^{\frac1p} \left(\alpha_p^{\frac1{p - 1}} + (1 - \alpha_p)^{\frac1{p - 1}}\right)^{1 - \frac1p} = C_p .
\end{align*}
Hence $c_p(\Omega, \mathcal{F}, \mathcal{G}, \mathbf{P}) = C_p$.} \hfill $\Box$
\end{example}

\begin{theorem}\label{ccond}
For every $p \in [1, \infty]$ and every $c \in [1, C_p]$, there exists a sub-$\sigma$-algebra $\mathcal{G} \subset \mathcal{L}$ such that
\begin{equation}\label{ccondp}
c_p([0, 1], \mathcal{L}, \mathcal{G}, \lambda) = c .
\end{equation}
\end{theorem}
\begin{proof}
Take any $\beta \in (0, 1)$ and consider the mapping
$$
J_\beta x = J_\beta(x) := \frac{\beta}{1 - \beta}\, (1 - x) , \quad x \in [\beta, 1] .
$$
It is cleat that $J_\beta$ is a homeomorphism of $[\beta, 1]$ onto $[0, \beta]$, and
$$
J_\beta^{-1} y = J_\beta^{-1}(y) = 1 - \frac{1 - \beta}{\beta}\, y , \quad y \in [0, \beta] .
$$
Let 
\begin{equation}\label{Gb}
\mathcal{G}_\beta := \left\{J_\beta(A) \cup A : \ A\subseteq [\beta, 1], \ A \in \mathcal{L}\right\} .
\end{equation}
It is easy to see that $\mathcal{G}_\beta$ is a sub-$\sigma$-algebra of $\mathcal{L}$ and that
every $\mathcal{G}_\beta$-measurable r.v. takes equal values at $x \in [\beta, 1]$ and $J_\beta x$. 
Then the condition
$$
\int_{J_\beta(A) \cup A} \xi(t)\, dt = \int_{J_\beta(A) \cup A} \mathbf{E}^{\mathcal{G}_\beta}\xi(t)\, dt \quad\mbox{for all}\quad
A\subseteq [\beta, 1], \ A \in \mathcal{L}
$$
implies that
\begin{align*}
& \mathbf{E}^{\mathcal{G}_\beta}\xi(x) = \mathbf{E}^{\mathcal{G}_\beta}\xi(J_\beta x) = (1 - \beta) \xi(x) + \beta \xi(J_\beta x) ,
\quad x \in [\beta, 1] , \\
& \mathbf{E}^{\mathcal{G}_\beta}\xi(y) = \mathbf{E}^{\mathcal{G}_\beta}\xi(J_\beta^{-1} y) = (1 - \beta) \xi(J_\beta^{-1} y) + \beta \xi(y) ,
\quad y \in [0, \beta] 
\end{align*}
for every $\xi \in L^1([0, 1], \mathcal{L}, \lambda)$. Hence
\begin{align}
\xi(y) - \mathbf{E}^{\mathcal{G}_\beta}\xi(y) &=  \xi(y) - (1 - \beta) \xi(J_\beta^{-1} y) - \beta \xi(y) \nonumber \\
&= (1 - \beta) (\xi(y) - \xi(J_\beta^{-1} y)) , \quad y \in [0, \beta] , \label{l} \\
\xi(x) - \mathbf{E}^{\mathcal{G}_\beta}\xi(x) &=  \xi(x) - (1 - \beta) \xi(x) - \beta \xi(J_\beta x) \nonumber \\
&= \beta (\xi(x) - \xi(J_\beta x)) , \quad x \in [\beta, 1]  \label{r} .
\end{align}

Suppose $1 < p < \infty$. Then
\begin{align*}
& \left\|\xi - \mathbf{E}^{\mathcal{G}_\beta}\xi\right\|_p^p = 
(1 - \beta)^p \int_0^\beta \left|\xi(y) - \xi(J_\beta^{-1} y)\right|^p dy \\
& + \beta^p \int_\beta^1 \left|\xi(x) - \xi(J_\beta x)\right|^p dx \\
& = (1 - \beta)^p \frac{\beta}{1 - \beta} \int_\beta^1 \left|\xi(x) - \xi(J_\beta x)\right|^p dx 
+ \beta^p \int_\beta^1 \left|\xi(x) - \xi(J_\beta x)\right|^p dx .
\end{align*}
Let $\kappa := \big(\beta\left((1 - \beta)^{p - 1} + \beta^{p - 1}\right)\big)^{1/p}$. Then it follows from the above that
\begin{align}
& \left\|\xi - \mathbf{E}^{\mathcal{G}_\beta}\xi\right\|_p = \kappa \left(\int_\beta^1 \left|\xi(x) - \xi(J_\beta x)\right|^p dx\right)^{1/p} \nonumber \\
& \le \kappa \left(\left(\int_\beta^1 \left|\xi(x) \right|^p dx\right)^{1/p} + \left(\int_\beta^1 \left|\xi(J_\beta x)\right|^p dx\right)^{1/p}\right)
\label{kappa} \\
& = \kappa \left(\left(\int_\beta^1 \left|\xi(x) \right|^p dx\right)^{1/p} + \left(\int_0^\beta \left|\xi(y)\right|^p dy\right)^{1/p}
\left(\frac{1 - \beta}{\beta}\right)^{1/p}\right) . \nonumber
\end{align}
Suppose $\|\xi\|_p = 1$ and let
$\gamma := \int_\beta^1 \left|\xi(x) \right|^p dx$.
Then 
$\int_0^\beta \left|\xi(y)\right|^p dy = 1 - \gamma$,
and
\begin{equation}\label{Phi}
\left\|\xi - \mathbf{E}^{\mathcal{G}_\beta}\xi\right\|_p 
\le \kappa \left(\gamma^{1/ p} + (1 - \gamma)^{1/ p}\left(\frac{1 - \beta}{\beta}\right)^{1/p}\right) =: \kappa\Psi_\beta(\gamma) .
\end{equation}
Solving $\Psi'_\beta(\gamma) = 0$, one gets
\begin{align}
& \gamma^{\frac1p - 1} - (1 - \gamma)^{\frac1p - 1}\left(\frac{1 - \beta}{\beta}\right)^{1/p} = 0  \nonumber \\ 
& \Longleftrightarrow \ \left(\frac{\gamma}{1 - \gamma}\right)^{\frac{1 - p}{p}} = \left(\frac{1 - \beta}{\beta}\right)^{1/p} \nonumber \\
& \Longleftrightarrow \ \gamma = \frac{\left(\frac{\beta}{1 - \beta}\right)^{\frac{1}{p - 1}}}{1 + \left(\frac{\beta}{1 - \beta}\right)^{\frac{1}{p - 1}}} 
= \frac{\beta^{\frac{1}{p - 1}}}{(1 - \beta)^{\frac{1}{p - 1}} + \beta^{\frac{1}{p - 1}}} \label{gamma} \\ 
& \Longleftrightarrow \ 1 - \gamma = \frac{(1 - \beta)^{\frac{1}{p - 1}}}{(1 - \beta)^{\frac{1}{p - 1}} + \beta^{\frac{1}{p - 1}}}\, .  \nonumber
\end{align}
So, $\Psi_\beta$ attains its maximum at $\gamma$ given by \eqref{gamma}, and, using the equality
$$
\frac{1}{p(p - 1)} + \frac1p = \frac{1}{p - 1}\, ,
$$
one gets
\begin{align}
\left\|\xi - \mathbf{E}^{\mathcal{G}_\beta}\xi\right\|_p 
&\le \kappa \left(\frac{\beta^{\frac{1}{p(p - 1)}}}{\left((1 - \beta)^{\frac{1}{p - 1}} + \beta^{\frac{1}{p - 1}}\right)^{\frac1p}}\right. \label{br} \\
& \quad \left. + \frac{(1 - \beta)^{\frac{1}{p(p - 1)}}}{\left((1 - \beta)^{\frac{1}{p - 1}} + \beta^{\frac{1}{p - 1}}\right)^{\frac1p}}
\left(\frac{1 - \beta}{\beta}\right)^{1/p}\right) \nonumber \\
&= \kappa\, 
\frac{\beta^{\frac{1}{p - 1}} + (1 - \beta)^{\frac{1}{p - 1}}}{\beta^{\frac1p}\left(\beta^{\frac{1}{p - 1}} + (1 - \beta)^{\frac{1}{p - 1}}\right)^{\frac1p}} \nonumber \\
&= \left((1 - \beta)^{p - 1} + \beta^{p - 1}\right)^{1/p} \left(\beta^{\frac{1}{p - 1}} + (1 - \beta)^{\frac{1}{p - 1}}\right)^{\frac{p - 1}p}  \nonumber
\end{align}
for all $\xi \in L^p([0, 1], \mathcal{L}, \lambda)$ with $\|\xi\|_p = 1$. Choosing $\xi$ such that $\int_\beta^1 \left|\xi(x) \right|^p dx$
equals $\gamma$ given by \eqref{gamma} and 
$$
\xi(J_\beta x) = - \left(\frac{(1 - \gamma)(1 - \beta)}{\gamma\beta}\right)^{1/p}\xi(x) ,
$$
one gets the equality $\|\xi\|_p = 1$ and equalities in \eqref{kappa} and \eqref{br}. Hence,
\begin{align*}
& c_p([0, 1], \mathcal{L}, \mathcal{G}_\beta, \lambda) = \left\|I - \mathbf{E}^{\mathcal{G}_\beta}\right\| \\
 &= \left((1 - \beta)^{p - 1} + \beta^{p - 1}\right)^{1/p} \left(\beta^{\frac{1}{p - 1}} + (1 - \beta)^{\frac{1}{p - 1}}\right)^{\frac{p - 1}p} 
= C_p(\beta)
\end{align*}
(see \eqref{Cp}). So, one can choose $\beta$ in such a way that $c_p([0, 1], \mathcal{L}, \mathcal{G}_\beta, \lambda) = c$ 
(cf. the end of Section \ref{intro}).

Suppose now $p = \infty$. It follows from \eqref{l}, \eqref{r} that
\begin{align*}
 \left\|\xi - \mathbf{E}^{\mathcal{G}_\beta}\xi\right\|_\infty &=
 \max\{\beta, (1 - \beta)\}\, \mathrm{ess}\!\!\sup_{x \in [\beta, 1]} |\xi(x) - \xi(J_\beta x)| \\
 &\le 2 \max\{\beta, (1 - \beta)\} \|\xi\|_\infty
\end{align*}
for all $\xi \in L^\infty([0, 1], \mathcal{L}, \lambda)$. If $\xi(J_\beta x) = - \xi(x)$, $x \in [\beta, 1]$, then
$$
 \left\|\xi - \mathbf{E}^{\mathcal{G}_\beta}\xi\right\|_\infty = 2 \max\{\beta, (1 - \beta)\} \|\xi\|_\infty .
$$
Hence 
$$
c_\infty([0, 1], \mathcal{L}, \mathcal{G}_\beta, \lambda) = \left\|I - \mathbf{E}^{\mathcal{G}_\beta}\right\|
= 2 \max\{\beta, (1 - \beta)\} .
$$
This proves \eqref{ccondp} for $c \in [1, 2)$. For $c = 2$, one can take $\mathcal{G} = \{\emptyset, [0, 1]\}$
and use \eqref{allelem}. 

Finally, suppose $p = 1$. Since the adjoint of the operator
$ I - \mathbf{E}^{\mathcal{G}_\beta} : L^1([0, 1]) \to L^1([0, 1])$
is the operator $ I - \mathbf{E}^{\mathcal{G}_\beta} : L^\infty([0, 1]) \to L^\infty([0, 1])$
(see, e.g., \cite[Theorem 6.1(vi)]{K02}), one has
\begin{align*}
c_1([0, 1], \mathcal{L}, \mathcal{G}_\beta, \lambda) & = 
\left\|I - \mathbf{E}^{\mathcal{G}_\beta}\right\|_{L^1([0, 1]) \to L^1([0, 1])} = \left\|I - \mathbf{E}^{\mathcal{G}_\beta}\right\|_{L^\infty([0, 1]) \to L^\infty([0, 1])} \\
& = c_\infty([0, 1], \mathcal{L}, \mathcal{G}_\beta, \lambda) = c .
\end{align*}
\end{proof}

The sub-$\sigma$-algebra $\mathcal{G}_\beta$ (see \eqref{Gb}) is close to the full $\sigma$-algebra $\mathcal{L}$ in the sense
that the $\sigma$-algebra generated by $\mathcal{G}_\beta$ and the set $[\beta, 1]$ coincides with $\mathcal{L}$.
It turns out that if a sub-$\sigma$-algebra $\mathcal{G}$ is much smaller than $\mathcal{L}$, then 
$c_p([0, 1], \mathcal{L}, \mathcal{G}, \lambda) = C_p$. More precisely, if $(\Omega, \mathcal{F}, \mathbf{P})$ is
a separable nonatomic probability space and there exists a r.v. $\xi$ on $(\Omega, \mathcal{F}, \mathbf{P})$, 
which is independent of a sub-$\sigma$-algebra $\mathcal{G} \subset \mathcal{F}$ and has a nontrivial Gaussian distribution, then 
$c_p(\Omega, \mathcal{F}, \mathcal{G}, \mathbf{P}) = C_p$, $1 \le p < \infty$ (see \cite{DP08}, \cite{F92}, or 
\cite[Definitions 1.5--1.7, Corollary 4.25, Corollary 6.12 and the paragraph following it]{PR13}).

\section{Estimates for compact operators on $L^p([0, 1])$ spaces}\label{3approx}

For a Banach space $X$, let $\mathfrak{F}(X)$ and $\mathcal{K}(X)$ denote the sets of bounded linear finite rank and compact linear operators 
on $X$, respectively.

\begin{definition}\label{def1} 
A Banach space $X$ is said to have the
{\sf bounded compact approximation property (BCAP)} if there
exists a constant $M \in (0, +\infty)$ such that given any $\varepsilon > 0$
and any finite set $F \subset X$, there exists an operator
$T \in \mathcal{K}(X)$ such that 
\begin{equation}\label{bcap}
\|I - T\| \le M \  \quad\mbox{and}\quad \  \|x - Tx\| < \varepsilon \quad\mbox{for all}\quad x \in F .
\end{equation}
We denote  by $M(X)$ the infimum of the constants $M$ for which the above conditions are satisfied.
\end{definition}

Many autors (see, e.g., \cite{C01}, \cite{CK}, \cite{LT1}, \cite{LT2}, and the references therein) have the condition $\|T\| \le M$ in place 
of $\|I - T\| \le M$ in the definition of BCAP and of related approximation properties. Let $m(X)$ be the infimum of the constants $M$ 
for which the conditions in this alternative definition of BCAP are satisfied. It is clear that 
$$
m(X) - 1 \le M(X) \le m(X) + 1 .
$$
If one is not interested in sharp constants, then it usually does not matter whether one knows $m(X)$ or $M(X)$. However,  the latter
appears naturally in estimates for the essential norms of operators by their measures of noncompactnes and it is desirable to know the
value of $M(X)$ (see \cite{AT87}, \cite{ES05}, \cite{LS}, \cite{S20}). It is well known that $m(L^p([0, 1])) = 1$, $1 \le p < \infty$ (see, e.g.,
\cite[Lemma 19.3.5]{P80}). The next result answers the question about the exact value of $M(L^p([0, 1]))$.

\begin{theorem}\label{approxT}
\begin{equation}\label{approx}
M(L^p([0, 1])) = C_p , \quad 1 \le p < \infty 
\end{equation}
(see \eqref{cp}, \eqref{Cp}).
\end{theorem}

The above result implies that $M(L^1([0, 1])) = 2$. This equality and $M(L^\infty([0, 1]))$ $= 2$ follow from the well known fact
that the spaces $L^1([0, 1])$ and $L^\infty([0, 1])$ have the so called Daugavet property, i.e. $\|I + T\| = 1 + ||T||$
for every $T \in \mathcal{K}(L^p([0, 1]))$, $p = 1$ or $\infty$ (see \cite{D63}, \cite{BP82}, \cite{L66}, and \cite[Ch. 6]{PR13}).

The proof of \eqref{approx} consists of proving the inequalities $M(L^p([0, 1])) \le C_p$ and $M(L^p([0, 1])) \ge C_p$. 
We prove the former with the help of Theorem \ref{upper} and derive the latter from an estimate for compact
operators on $L^p([0, 1])$ (see Theorem \ref{gammaT}), which we think might be of an independent interest. 

\begin{lemma}\label{gammaL}
Let $1 \le p < \infty$, $\delta \ge 0$, $\gamma \in \mathbb{C}$, and let $T \in \mathcal{K}(L^p([0, 1]))$ be such that 
$$
\|\gamma\mathbbm{1} - T\mathbbm{1}\|_{L^p} \le \delta .
$$
Then
\begin{equation}\label{delta}
\|I - T\|_{L^p \to L^p} \ge \|I -  \gamma\mathbf{ E}\|_{L^p \to L^p} - \delta .
\end{equation}
\end{lemma}
\begin{proof}
Take an arbitrary $\varepsilon > 0$. Since $T \in \mathcal{K}(L^p)$, there exists $K_0 \in \mathfrak{F}(L^p)$ such that 
$\|T - K_0\|_{L^p \to L^p} \le \varepsilon$ (see, e.g., \cite[Sections 1a and 1e]{LT1}). The operator $K_0$ admits the following representation
$$
K_0 f = \sum_{j = 1}^N \left(\int_0^1 g_j(t) f(t)\, dt \right) h_j , \ \ \ \forall f \in L^p([0, 1]) ,
$$
where $g_j \in L^{p'}([0, 1])$, $h_j \in L^p([0, 1])$, $j = 1, \dots, N$, $N \in \mathbb{N}$. 
Approximating the functions $g_j$ by simple functions and rearranging the terms,
one can construct an operator $K_1 \in \mathfrak{F}(L^p)$ such that $\|K_0 - K_1\|_{L^p \to L^p} \le \varepsilon$ and
$$
K_1 f = \sum_{k = 1}^M \left(\int_{A_k}  f(t)\, dt \right) \varphi_k , \ \ \ \forall f \in L^p([0, 1]) ,
$$
where $\varphi_k  \in L^p([0, 1])$, $k = 1, \dots, M$, $M \in \mathbb{N}$, and $A_k$ are pairwise disjoint measurable subsets of $[0, 1]$ 
of positive measure such that $[0, 1] = \cup_{k = 1}^M A_k$.

Let $\nu_p(\gamma) := \|I - \gamma \mathbf{ E}\|_{L^p([0, 1]) \to L^p([0, 1])}$. There exists $\chi \in L^p([0, 1])$ such that 
$\|\chi\|_{L^p} =1$ and $\|\chi - \gamma \mathbf{E} \chi\|_{L^p} \ge \nu_p(\gamma) - \varepsilon$. Let $\kappa := \mathbf{E}\chi$.
The probability space $(A_k, \mathcal{L}_k, \mathbf{P}_k)$, where
$$
\mathcal{L}_k := \{A\cap A_k : \ A \in \mathcal{L}\} , \quad \mathbf{P}_k := \frac{1}{\lambda(A_k)}\, \lambda ,
$$
is isomorphic (modulo sets of measure $0$) to $([0, 1], \mathcal{L}, \lambda)$ (see, e.g., \cite[Theorem 9.2.2 and Corollary 6.6.7]{B07}). 
Let $w_k : A_k \to [0, 1]$ be such an isomorphism and let $\chi_k := \chi\circ w_k$. Then
$\|\chi_k\|_{L^p(A_k, \mathbf{P}_k)} =1$, $\|\chi_k - \gamma \mathbf{E}_k \chi_k\|_{L^p(A_k, \mathbf{P}_k)} \ge \nu_p(\gamma) - \varepsilon$,
and $\mathbf{E}_k\chi_k = \kappa$, where $\mathbf{E}_k\xi := \int_{A_k} \xi(t)\, d\mathbf{P}_k(t) 
= \frac{1}{\lambda(A_k)}\int_{A_k} \xi(t)\, dt$. Finally, define $\psi$ by $\psi(t) = \chi_k(t)$, $t \in A_k$, $k = 1, \dots, M$. 

Since
$$
\int_{A_k} \left(\chi_k(t) - \kappa\right)\, dt = \kappa\lambda(A_k)  - \kappa\lambda(A_k) = 0 ,
$$
one gets $K_1\left(\psi - \kappa \mathbbm{1}\right) = 0$ and $\mathbf{E}\psi = \kappa$.  Hence
\begin{align*}
& \|(I -K_1)\psi\|_{L^p} = \|\psi -K_1\psi\|_{L^p} = \left\|\psi -K_1\left(\kappa \mathbbm{1}\right)\right\|_{L^p} \\
& \ge \left\|\psi - T\left(\kappa \mathbbm{1}\right)\right\|_{L^p} - 2\varepsilon \left\|\kappa \mathbbm{1}\right\|_{L^p} \ge
\left\|\psi - \gamma\kappa \mathbbm{1}\right\|_{L^p} - |\kappa|\|\gamma\mathbbm{1} - T\mathbbm{1}\|_{L^p} - 2\varepsilon |\kappa| \\
& \ge \left(\sum_{k = 0}^M \int_{A_k} \left|\chi_k(t) - \gamma\kappa\right|^p\, dt\right)^{1/p} - (\delta + 2\varepsilon) |\kappa| \\
& = \left(\sum_{k = 0}^M \int_{A_k} \left|\chi_k(t) - \gamma\mathbf{E}_k\chi_k\right|^p\, dt\right)^{1/p} - (\delta + 2\varepsilon) |\kappa| \\
& \ge \left(\sum_{k = 0}^M (\nu_p(\gamma) - \varepsilon)^p\int_{\Omega_k} \left|\chi_k(t)\right|^p\, dt\right)^{1/p} - (\delta + 2\varepsilon) \|\psi\|_{L^p} \\
& = (\nu_p(\gamma) - \delta - 3\varepsilon) \|\psi\|_{L^p} .
\end{align*}
So, $\|I - K_1\|_{L^p \to L^p} \ge \nu_p(\gamma) - \delta - 3\varepsilon$ and hence
$$
\|I - T\|_{L^p \to L^p} \ge \nu_p(\gamma) - \delta - 5\varepsilon , \ \ \ \forall \varepsilon > 0 ,
$$
i.e. \eqref{delta} holds.
\end{proof}

\begin{theorem}\label{gammaT}
Let $1 \le p < \infty$, $\gamma \in \mathbb{C}$, and let $T \in \mathcal{K}(L^p([0, 1]))$. Then 
\begin{equation}\label{delta1}
\|I - T\|_{L^p \to L^p} + \inf_{\|u\|_{L^p} = 1} \|(\gamma I - T)u\|_{L^p} \ge \|I -  \gamma\mathbf{ E}\|_{L^p \to L^p} .
\end{equation}
In particular,
\begin{equation}\label{gamma1}
\|I - T\|_{L^p \to L^p} + \inf_{\|u\|_{L^p} = 1} \|(I - T)u\|_{L^p} \ge \|I - \mathbf{ E}\|_{L^p \to L^p} = C_p 
\end{equation}
(see \eqref{OpC} and \eqref{cp}).
\end{theorem}
\begin{proof}
Take an arbitrary $\varepsilon > 0$. Let
$$
\delta := \inf_{\|u\|_{L^p} = 1} \|(\gamma I - T)u\|_{L^p} .
$$
There exists $u_0  \in L^p([0, 1])$ such that $\|u_0\|_{L^p} = 1$ and $\|(\gamma I - T)u_0\|_{L^p} < \delta + \epsilon$. Then
there exists an approximation $v  \in L^p([0, 1])$ of $u_0$ such that $v \not= 0$ almost everywhere in $[0, 1]$ and
$$
\|\gamma v - Tv\|_{L^p} \le (\delta + 2\varepsilon) \|v\|_{L^p} .
$$
Let $v_0 := v/\|v\|_{L^p}$ and
$$
w(t) :=  \int_0^t |v_0(x)|^p dx .
$$
Then $w$ is a strictly increasing absolutely continuous function that maps $[0, 1]$ onto itself. Consider the operator $J$ defined by
$$
(J f)(t) := v_0(t) f(w(t)) , \ \ \ t \in [0, 1] .
$$
It is easy to see that $J$ is an isometric automorphism of $L^p$ and $J\mathbbm{1} = v_0$.

Set $T_0 := J^{-1}TJ  \in \mathcal{K}(L^p)$. Then 
$$
\|\gamma\mathbbm{1} - T_0\mathbbm{1}\|_{L^p} = \|\gamma J^{-1}v_0 - J^{-1}Tv_0\|_{L^p} 
= \|\gamma v_0 - Tv_0\|_{L^p}\le \delta + 2\varepsilon ,
$$
and it follows from Lemma \ref{gammaL} applied to $T_0$ that
\begin{align*}
\|I - T\|_{L^p \to L^p} &= \|JJ^{-1} - JT_0J^{-1}\|_{L^p \to L^p} = \|I - T_0\|_{L^p \to L^p} \\
&\ge \|I -  \gamma\mathbf{ E}\|_{L^p \to L^p} - \delta - 2\varepsilon  \quad\mbox{for all}\quad \varepsilon > 0 .
\end{align*}
\end{proof}
Theorem \ref{gammaT} remains valid for narrow operators $T \in \mathcal{B}(L^p([0, 1]))$ (\cite{SS2}; see \cite{PR13} for information on narrow operators).

\begin{corollary}\label{eigen}
Let $1 \le p < \infty$. If $\gamma \in \mathbb{C}$ is an eigenvalue of $T \in \mathcal{K}(L^p([0, 1]))$, then 
\begin{equation}\label{delta0}
\|I - T\|_{L^p \to L^p} \ge \|I -  \gamma\mathbf{ E}\|_{L^p \to L^p} .
\end{equation}
In particular, if $I - T$ is not invertible, then
\begin{equation}\label{gamma0}
\|I - T\|_{L^p \to L^p} \ge C_p .
\end{equation}
\end{corollary}

\begin{proof}[Proof of Theorem \ref{approxT}]
It follows from Theorem \ref{gammaT} that for every operator $T \in \mathcal{K}(X)$ satisfying 
the second inequality in \eqref{bcap} the following estimate holds
$$
\|I - T\|_{L^p \to L^p} + \varepsilon \ge  C_p .
$$
Hence $M(L^p([0, 1])) \ge C_p$. 

To prove the opposite inequality, take any $\varepsilon > 0$ and any finite set $\{f_1, \dots, f_N\}$ $\subset$ $L^p([0, 1])$.
There exist a partition of $[0, 1]$ into pairwise disjoint measurable sets $A_k$, $k = 1, \dots, M$  
of positive measure and simple functions $g_1, \dots, g_N$ that are constant on each $A_k$ and satisfy the inequalities
$\|f_n - g_n\|_{L^p([0, 1])} < \varepsilon/2$, $n = 1, \dots, N$. Let $\mathcal{G}$ be the sub-$\sigma$-algebra of $\mathcal{L}$
generated by the sets $A_k$, $k = 1, \dots, M$
and consider the conditional expectation operator  $\mathbf{E}^\mathcal{G} : L^p([0, 1]) \to L^p([0, 1])$. The range of 
$\mathbf{E}^\mathcal{G}$ is the linear span of the indicator functions of the sets $A_k$ and hence is an $M$ dimensional
linear subspace of $L^p([0, 1])$. So, $\mathbf{E}^\mathcal{G} \in \mathfrak{F}(L^p([0, 1])) \subset\mathcal{K}(L^p([0, 1]))$.
According to Theorem \ref{upper}, $\left\|I - \mathbf{E}^\mathcal{G}\right\|_{L^p \to L^p} \le C_p$. Further, 
$\mathbf{E}^\mathcal{G} g_n = g_n$ by construction, and
\begin{align*}
\left\|f_n - \mathbf{E}^\mathcal{G} f_n\right\|_{L^p} &\le \left\|f_n - g_n\right\|_{L^p} 
+ \left\|g_n - \mathbf{E}^\mathcal{G} f_n\right\|_{L^p} < \frac{\varepsilon}{2} + \left\|\mathbf{E}^\mathcal{G}(g_n -  f_n)\right\|_{L^p} \\
&\le \frac{\varepsilon}{2} + \left\|g_n -  f_n\right\|_{L^p} < \varepsilon , \quad  n = 1, \dots, N ,
\end{align*}
since $\left\|\mathbf{E}^\mathcal{G}\right\|_{L^p \to L^p} = 1$. Hence $M(L^p([0, 1])) \le C_p$.
\end{proof}

All results of this section remain true for $L^p(\Omega, \mathcal{B}, \mu)$, where
$\Omega$ is a complete separable metric space, $\mathcal{B}$ is the Borel $\sigma$-algebra of $\Omega$, and $\mu$
is a nonatomic finite measure, since $(\Omega, \mathcal{B}, \frac{1}{\mu(\Omega)}\mu)$ is isomorphic, 
modulo sets of measure $0$, to $([0, 1], \mathcal{L}, \lambda)$ (see, e.g., \cite[Theorem 9.2.2]{B07}).

\end{document}